\documentclass[12pt,a4paper,british]{scrartcl}
\usepackage[latin1]{inputenc}
\usepackage[british]{babel}
\usepackage{amsmath,amssymb,leftidx}
\usepackage[amsmath,thref,thmmarks]{ntheorem}

\theoremstyle{plain}
\theoremheaderfont{\normalfont\bf}
\theorembodyfont{\itshape}
\theoremseparator{.}
\theoremnumbering{arabic}
\theoremsymbol{}
\newtheorem{thm}{Theorem}
\newtheorem{lem}[thm]{Lemma}

\newtheorem{prop}[thm]{Proposition}

\theoremstyle{plain}
\theoremheaderfont{\normalfont\bf}
\theorembodyfont{\upshape}
\theoremseparator{.}
\theoremsymbol{}

\newtheorem{example}[thm]{Example}

\theoremstyle{break}
\theoremheaderfont{\normalfont\bf}
\theorembodyfont{\upshape}
\theoremseparator{:}
\theoremsymbol{}
\newtheorem{defi}[thm]{Definition}

\theoremstyle{nonumberplain}
\theoremseparator{:}
\theoremheaderfont{\bf}
\theorembodyfont{\upshape}
\theoremsymbol{{\Large\ensuremath{_\dashv}}}
\newtheorem{proof}{Proof}

\newcommand{\N}{\mathbb{N}}
\renewcommand{\phi}{\varphi}
\renewcommand{\theta}{\vartheta}
\newcommand{\<}{\langle}
\renewcommand{\>}{\rangle}
\newcommand{\al}{\alpha}
\newcommand{\be}{\beta}
\newcommand{\DD}[1]{\mathbb{D}_{2^{#1}}}
\newcommand{\QFSL}{\mathit{QFSL}}
\renewcommand{\|}{\,|\,}
\newcommand{\StarN}{{}^{*}\mathbb{N}}
\newcommand{\Star}[1]{{}^{*}{#1}}
\newcommand{\restr}{\upharpoonright}
\newcommand{\std}{{}^{\circ}}
\newcommand{\C}{\mathcal{C}}
\newcommand{\Dcal}{\mathcal{D}}
\newcommand{\Bcal}{\mathcal{B}}

\newcommand{\W}{\mathcal{W}}

\renewcommand{\epsilon}{\varepsilon}
\newcommand{\SL}{\mathit{SL}}
\newcommand{\hatDD}[1]{\widehat{\mathbb{D}}_{#1}}
\newcommand{\p}{\mathfrak{p}}

\newcommand{\Iff}{\Leftrightarrow}

\title{Predicate Exchangeability and Language Invariance in Pure Inductive Logic\footnote{Submitted to the Proceedings of the
		1st Reasoning Club Meeting, eds. J.P. van Bendegem, J.Murzi, University
		Foundation, Brussels, 2012, appearing in Logique et Analyse.}}
\author{M.S.Klie{\ss}
	and J.B.Paris\\
	School of Mathematics\\
	The University of Manchester\\
	Manchester M13 9PL\\
\small{malte.kliess@postgrad.manchester.ac.uk},\vspace{-2ex}\\
  \small{jeff.paris@manchester.ac.uk}}
\date{\today}

\begin{document}
\maketitle
\begin{abstract}
	\noindent In Pure Inductive Logic, the rational principle of Predicate Exchangeability states that permuting the predicates
	in a given language $L$ and replacing each occurrence of a predicate in an $L$-sentence $\phi$ according to this
	permutation should not change our belief in the truth of $\phi$. In this paper we study when a probability
	function $w$ on a purely unary language $L$ satisfying Predicate Exchangeability also satisfies the principle of
	Unary Language Invariance.
\end{abstract}


\section{Introduction}

In the study of logical probability in the sense of Carnap's Inductive Logic programme, \cite{Carnap3}, \cite{Carnap4}, the notion of {\it symmetry} plays a leading role.
In the assignment of beliefs, as subjective probabilities, it seems logical, or rational, to observe prevailing symmetries, a typical example being the perceived fairness
of a coin toss, at least in the absence of any inside knowledge to the contrary. For this reason a number of rational principles have been proposed in Inductive Logic which
are based on invariance under various notions of symmetry, principles which it is argued a choice of logical or rational (we use these two words synonymously) probability
function should satisfy. The most prevailing of these, accepted by both the founding fathers of Inductive Logic, W.E. Johnson \cite{Johnson}, and
Rudolf Carnap \cite{CarnapZ}, is that the names we give things, in particular constants and predicates, should not matter when it comes to assigning probabilities.
Thus, since interchanging which side of the coin we call heads and which we call tails does not change what we understand by a coin toss, both outcomes should
rationally receive the same probability.

A second, ubiquitous, rational principle is that when assigning rational probabilities `irrelevant information' can be disregarded. Indeed the central principle of Johnson
and Carnap, the so called {\it Johnson's Sufficientness Postulate}, is just such an example. Just as with saying what exactly we might mean by a `symmetry' this directive
does of course raise the question of what exactly we mean by an `irrelevance information', and numerous interpretations have been mooted, generally based on the
idea that such information is expressed in a disjoint, or partially disjoint language.

A third, more recent and rather overarching, rational principle is the requirement of {\it language invariance}. By that we mean that to be rational a probability function should
not be restricted to one special language but be extendable to larger languages, and furthermore that those additional rational principles which we imposed in the
context of  the original language should also be satisfied by these extensions.

In this paper we shall study two symmetry principles, {\it Constant Exchangeability}\footnote{Johnson's {\it Permutation Postulate} and Carnap's  {\it Axiom of Symmetry}.}
and {\it Predicate Exchangeability}, in the presence of language invariance with the main goal of providing a representation theorem along the lines of
de Finetti's Representation Theorem for Constant Exchangeability alone, see for example \cite{Finetti}, \cite{Continuum}. Although rather technical, at least in relation
to the seemingly elementary mathematics at the heart of Inductive Logic, such results have, starting with Gaifman \cite{Gaifman2} and Humburg \cite{Humburg}, been
an extremely powerful tool in our understanding of the interrelationship between the various rational principles which have been proposed. Hopefully the results given
here will also find similar applications in the future.

The structure of this paper is as follows. In Section 2 we shall introduce the notation and give precise formulations of the main principles we shall be studying. In Section 3
we shall provide a representation theorem for probability functions satisfying language invariance with Constant and Predicate Exchangeability assuming a particularly
strong irrelevance condition, the {\it Constant Irrelevance Principle}, and in the next section show a similar result without this assumption. This latter representation
theorem shows that all such probability functions are in a sense convex mixtures of probability functions satisfying the so called {\it Weak Irrelevance Principle}, and
conversely. Finally in Section 5 we will give a general representation theorem for probability functions satisfying Constant and Predicate Exchangeability alone, showing
that they are mixtures (not necessarily convex) of such probability functions which additionally satisfy language invariance.

The philosophical standpoint of this paper is  {\it Pure Inductive Logic}, see \cite{Continuum}, \cite{ParisVencovskaBook}, a branch of Carnap's Inductive Logic which
he already described in \cite{CarnapZ}. Thus we shall be interested in studying logical probability without relation to specific interpretations. Of course the rational
principles one proposes may have their genesis in real world examples but once a principle is formulated it is studied in Pure Inductive Logic through the agency
of mathematics. The subsequent interest within philosophy lies, we would opine, mainly in  considering what these mathematical conclusions tell us about the
original and like motivating examples.

\section{Notation and Principles}
We will be working in the usual context of (unary) Pure Inductive Logic. Thus the first order languages we will be concerned with consist only of
finitely many unary predicate symbols $P_i$ and countably many constant symbols\footnote{For convenience, we shall henceforth refer to these just as `predicates' and `constants'.}
$a_1,a_2,\dotsc,a_m,\dotsc$, which should be thought
of as exhausting the universe. We will write $L_q$ to indicate the language containing just the predicates $P_1,\dotsc,P_q$.
Let $\SL$ denote the set of sentences of the language $L$, $\QFSL$ the set of quantifier-free sentences of $L$.

An \emph{atom} $\al(x)$ of $L$ is a formula
\begin{gather*}
	P_1^{\epsilon_1}(x)\wedge P_2^{\epsilon_2}(x)\wedge\dotsb\wedge P_q^{\epsilon_q}(x),
\end{gather*}
with $\epsilon_i\in\{0,1\}$ and $P_i^1(x)$, $P_i^0(x)$ standing for $P_i(x)$, $\neg P_i(x)$, respectively.\footnote{In the literature,
	the notation $\pm P_i(x)$ is more common; however, in the scope of this paper, the notation $P_i^{\epsilon_i}(x)$ is more convenient.} Note that
for $L$ containing $q$ predicates there are $2^q$ atoms, which we shall denote $\al_1, \dotsc, \al_{2^q}$.

A \emph{state description} of $L$ for\footnote{The entries in such lists will be taken to be distinct unless otherwise stated.} $a_{i_1}, \ldots, a_{i_n}$  is a sentence
\begin{gather*}
	\Theta(a_{i_1},\dotsc,a_{i_n}) = \bigwedge_{j=1}^n \al_{h_j}(a_{i_j}),
\end{gather*}
where $h_j\in\{1,\dotsc,2^q\}$ for $j=1,\dotsc,n$.

A \emph{probability function on $L$} is a function $w:\SL \rightarrow [0,1]$ satisfying the following conditions for all $\theta, \phi, \exists x \, \psi(x)\in\SL$:
\begin{itemize}
	\item[(P1)] If $\models\theta$, then $w(\theta) = 1$.
	
	\item[(P2)] If $\theta\models\neg\phi$, then $w(\theta\vee\phi) = w(\theta) + w(\phi)$.
	
	\item[(P3)] $w(\exists x\, \psi(x)) = \lim_{n\rightarrow\infty} w(\bigvee_{j=1}^n \psi(a_j))$.
\end{itemize}

The following theorem will allow us to restrict our studies to quantifier-free sentences.

\begin{thm}[Gaifman, \cite{Gaifman}]\label{Theorem1}
	Let $w:\QFSL \rightarrow [0,1]$ be a function satisfying (P1), (P2) for all $\theta, \phi\in\QFSL$. Then there exists a unique $w':\SL\rightarrow[0,1]$ satisfying
	(P1)-(P3) extending $w$.
\end{thm}

Since any quantifier-free sentence of $L$ is logically equivalent to a disjunction of state descriptions, by (P2) and Theorem \ref{Theorem1} a probability
function is determined by its values on the state descriptions. Let
\begin{gather*}
	\vec{x}\in\DD{q} := \left\{\<x_1,\dotsc,x_{2^q}\>\| x_i\geq 0, \sum_{i=1}^{2^q} x_i = 1\right\}.
\end{gather*}
Then we obtain an example of a probability function by defining $w_{\vec{x}}$ on state descriptions via
\begin{gather*}
	w_{\vec{x}}(\Theta(a_{i_1},\dotsc,a_{i_n})) := \prod_{i=1}^{2^q} x_i^{n_i},
\end{gather*}
where $n_i = |\{j\| h_j = i\}|$.

These functions are quite important examples, as they form the building blocks in de Finetti's Representation Theorem. Before
stating this theorem, we need to introduce the Principle of Constant Exchangeability:

{\bfseries The Principle of Constant Exchangeability, Ex}\\
{\itshape A probability function $w$ on $\SL$ satisfies \emph{Constant Exchangeability} if for each\\
$\phi(a_1,\dotsc,a_{n})\in\SL$, and $\sigma$ a permutation of $\N^{+}\, (= \{1,2,3, \ldots\})$,
\begin{gather*}
	w(\phi(a_1,\dotsc,a_n)) = w(\phi(a_{\sigma(1)},\dotsc,a_{\sigma(n)})).
\end{gather*}}

Notice that the $w_{\vec{x}}$ satisfy Ex. Ex is such a well accepted principle in Inductive Logic that we shall henceforth take it as a standing assumption
throughout that all the probability functions we consider satisfy it.

We shall therefore not mention the particular constants whenever they are understood from the context.

\begin{thm}[de Finetti's Representation Theorem]
	Let $L=L_q$ and $w$ be a probability function on $\SL$ satisfying Ex. Then there exists a normalized, $\sigma$-additive measure $\mu$ on the
	Borel subsets of $\DD{q}$ such that
	\begin{gather}
		w\left(\bigwedge_{j=1}^n \al_{h_j}(a_j)\right) = \int_{\DD{q}} w_{\vec{x}}\left(\bigwedge_{j=1}^n \al_{h_j}(a_j)\right)\,d\mu(\vec{x}).\label{eq1.1}
	\end{gather}
	Conversely, given such a measure $\mu$, the function $w$ defined by \eqref{eq1.1} is a probability function on $\SL$ satisfying Ex.
\end{thm}

It is straightforward to show (see \cite{ParisVencovskaBook}) that these $w_{\vec{x}}$ are characterized as those probability functions which satisfy Ex together with

{\bfseries The Principle of Constant Irrelevance, IP}\\
{\itshape A probability function $w$ on $\SL$ satisfies \emph{Constant Irrelevance} if for $\theta, \phi \in QFSL$\\ with no constants in common,
\begin{gather*}
	w(\theta \wedge \phi) = w(\theta)\cdot w(\phi).
\end{gather*}}

Thus de Finetti's Representation Theorem can be alternately stated as saying that every probability function satisfying Ex is a convex mixture of probability
functions satisfying IP, and conversely.

The principles that are of particular interest to us in this paper are:

{\bfseries The Principle of Predicate Exchangeability, Px}\\
{\itshape A probability function $w$ on $\SL$ satisfies \emph{Predicate Exchangeability} if whenever $\phi\in\SL$ and $\phi'$ is the result of replacing the predicates\footnote{In such lists we shall always assume that the members are distinct.
}
$P_{i_1},\dotsc,P_{i_m}$ in $\phi$ by $P_{k_1},\dotsc,P_{k_m}$, then
\begin{gather*}
	w(\phi) = w(\phi').
\end{gather*}}

{\bfseries The Principle of Unary Language Invariance, ULi}\\
{\itshape A probability function $w$ on $\SL$ satisfies \emph{Unary Language Invariance} if there exists a family of probability functions $w^{\mathcal{L}}$,
one for each finite (unary) language $\mathcal{L}$, satisfying Px (and by standing assumption Ex), such that $w=w^L$ and whenever $\mathcal{L}'\subseteq\mathcal{L}$, then $w^{\mathcal{L'}}=
w^{\mathcal{L}}\restr S\mathcal{L}'$, the restriction of $w$ to the sentences of $\mathcal{L}'$.

We say that $w$ satisfies ULi \emph{with $\mathcal{P}$} (for some principle $\mathcal{P}$), if each of the functions $w^{\mathcal{L}}$ satisfy $\mathcal{P}$.}

Notice that if $w^{\mathcal{L}}, w^{\mathcal{L'}}$ are members of a language invariant family and $\mathcal{L},\mathcal{L'}$ have the same number of predicates then
$w^{\mathcal{L}}$ is the same as  $w^{\mathcal{L'}}$ up to renaming predicates. For that reason it will, for the most part, be enough for us to focus our attention
on the members $w^{\mathcal{L}}$ of the family when $\mathcal{L}=L_q$ for some $q$.

This also illustrates the motivation for pairing ULi with Px; for if we were to drop Px from the
definition, then $w^{\mathcal{L}}$ would depend on the particular set of predicates
in $\mathcal{L}$, and we would be imposing some a priori semantics on the languages.\footnote{In fact, as one easily checks, without Px, all of the $w_{\vec{x}}$ functions
	can be extended to obtain a language invariant family, and the choices are arbitrary on every level, which makes Language
	Invariance in this form a trivial statement.}

Given a permutation $\sigma$ of the predicates of $L$, there is a unique permutation of the atoms of $L$ that is induced by $\sigma$:
For $\al(x) = \bigwedge_{i=1}^q P_i^{\epsilon_i}(x)$ an atom of $L$, let $\sigma\al(x)$ be the atom given by
\begin{gather*}
	\sigma\al(x) = \bigwedge_{i=1}^q \sigma(P_i)^{\epsilon_i}(x).
\end{gather*}
This now in turn induces a permutation on $\SL$ in the obvious way. Abusing notation, we identify these permutations of atoms and $L$-sentences with $\sigma$.
We shall write \emph{$\sigma$ is induced by Px} to indicate that $\sigma$ arises from a permutation of predicates.

\section{A First Representation Theorem}

Since the $w_{\vec{x}}$ are the building blocks for probability functions satisfying Ex (see de Finetti's Theorem above), these functions are of special interest to us.
We will therefore begin by studying when they satisfy ULi, equivalently when probability functions satisfying Ex and IP satisfy ULi.

Suppose a probability function $w$ on some language $L$ satisfied Predicate Exchangeability. Then the probability that $w$ assigns
any atom $\al$ of $L$ only depends on the number of predicates in $\al$ that occur negated.\footnote{This is an arbitrary choice. One could also
	count the number of predicates that occur positively in $\al$, as the argument is symmetrical.}
To see this notice that if $\al, \al'$ are atoms then $\al'$ can be obtained from $\al$ by a permutation of predicates just if both atoms have the same number of negated predicates.

It is thus convenient to introduce a function assigning each atom the corresponding number of
predicates:
\begin{defi}
	Let $L=L_q$. Define $\gamma_q:\{1,\dotsc,2^q\}\rightarrow\{0,\dotsc,q\}$ by
	\begin{gather*}
		\gamma_q(i) = k \Iff \al_i\text{ contains $k$ negated predicates.}
	\end{gather*}
	We shall drop the index $q$ whenever it is understood from the context.
\end{defi}

Now considering $\vec{c}\in\DD{q}$ it follows that $w_{\vec{c}}$ satisfies Predicate Exchangeability if and only if
$c_i = c_j$ whenever $\gamma(i)=\gamma(j)$. With this in mind we shall assume that our enumeration of the atoms is such that the number of
negated predicates is non-decreasing as we move right through $\al_1, \al_2, \dotsc,\al_{2^q}$. Since for each $i\in\{0,\dotsc,q\}$ there are
$\binom{q}{i}$ atoms of $L_q$ with $i$ predicates occurring negatively we therefore have that for $w_{\vec{c}}$ satisfying Px
\begin{gather*}
	\vec{c} = \<\C_0,\C_1,\dotsc,\C_1,\C_2,\dotsc,\C_2,\dotsc,\C_{q-1},\dotsc,\C_{q-1},\C_q\>,
\end{gather*}
i.e. $c_i = \C_{\gamma(i)}$ for $i=1,2, \ldots, 2^q$, and
\begin{gather*}
	\sum_{i=0}^q \binom{q}{i}\C_i = 1.
\end{gather*}
Thus any such $\vec{c}$ gives us a unique $\vec{\C} = \<\C_0,\C_1,\C_2,\dotsc,\C_q\>$ with the properties
\begin{gather*}
	\forall i\in\{0,\dotsc,q\}\,\, \C_i\geq 0 \text{ and } 1 = \sum_{i=0}^q \binom{q}{i}\C_i.
\end{gather*}
Conversely, any $\vec{\C}$ with these properties provides a unique $\vec{c}\in\DD{q}$ such that $w_{\vec{c}}$ satisfies Px, giving us
a 1-1 correspondence between these $\vec{c} \in \DD{q}$  and the elements of
\begin{gather}
	\hatDD{q} := \left\{\vec{\C} = \<\C_0,\C_1,\C_2,\dotsc,\C_q\>\| \forall i\in\{0,\dotsc,q\}\,\C_i\geq 0 \text{ and } 1 = \sum_{i=0}^q \binom{q}{i}\C_i\right\}.\label{eq2.1}
\end{gather}
We shall refer to elements of the set above as the \emph{alternative notation} for such a $\vec{c}\in\DD{q}$.

\vspace{1ex} Given an atom $\al$ of $L_q$, we can view this atom as a quantifier-free sentence in the extended language $L_{q+1}$, and obtain
\begin{gather*}
	\al (x) \equiv \al^{+}(x) \vee \al^{-}(x) = \left(\al(x) \wedge P_{q+1}(x)\right) \vee \left(\al(x) \wedge \neg P_{q+1}(x)\right).
\end{gather*}
Now suppose $\vec{c}\in\DD{q}$, $\vec{d}\in\DD{q+1}$ are such that $w_{\vec{d}}\restr \SL_q = w_{\vec{c}}$ and both satisfy Px.
Then by the logical equivalence given above, we must have
\begin{gather*}
	w_{\vec{c}} (\al) = w_{\vec{d}}\,(\al) = w_{\vec{d}}\,(\al^{+}) + w_{\vec{d}}\,(\al^{-}).
\end{gather*}
Suppose $\vec{\C}\in\hatDD{q}$, $\vec{\Dcal}\in\hatDD{q+1}$ are the corresponding alternative notations for $\vec{c}$ and $\vec{d}$. Then we
obtain for each $i\in\{0,\dotsc,q\}$,
\begin{gather*}
	\C_i = \Dcal_i + \Dcal_{i+1}.
\end{gather*}
The following proposition generalizes this to ULi families.

\begin{prop}
	Let $w_{\vec{c}}$ be a probability function on $L_q$. Suppose $w_{\vec{c}}$ is a member of a ULi with IP family $\W$
	and assume $w_{\vec{d}}\in\W$ is a probability function on $L_r$ for some $r>q$. Let $\vec{\C}, \vec{\Dcal}$ be the corresponding
	alternative notations for $\vec{c}$, $\vec{d}$. Then for each $j\in\{0,\dotsc,q\}$, we have
	\begin{gather}
		\C_j = \sum_{k=j}^{r-q+j}\binom{r-q}{k-j}\Dcal_k.\label{eq2.2}
	\end{gather}
\end{prop}

\begin{proof}
	We show this by induction on $s := r - q$. In case $s = 1$, we have for each $j\in\{0,\dotsc,q\}$,
	\begin{gather*}
		\C_j = \Dcal_j + \Dcal_{j+1},
	\end{gather*}
	since for $\al$ an atom of $L_q$ with $j$ negated predicates, we have in $L_r$ ($=L_{q+1}$)
	\begin{gather*}
		\al = \al^{+} \vee \al^{-},
	\end{gather*}
	where $\al^{+}, \al^{-}$ are atoms of $L_r$ with $j, j+1$ negated predicates, respectively.
	
	Now let $s = p + 1$ and assume the result holds for $p$. Let $\Dcal_i'$ denote the corresponding values for the atoms of $L_{q+p}$.
	By the inductive hypothesis we have
	\begin{gather*}
		\C_j = \sum_{k=j}^{(q + p) - q + j}\binom{(q+p)-q}{k-j}\Dcal_k'.
	\end{gather*}
	Just as in the case $s = 1$ we have $\Dcal_k' = \Dcal_k + \Dcal_{k+1}$ for each $0\leq k\leq q+p$, so we obtain
$$		\C_j =
		 \sum_{k=j}^{p+j}\binom{p}{k-j}(\Dcal_k + \Dcal_{k+1})
		= \sum_{k=j}^{p+1+j}\binom{p+1}{k-j}\Dcal_k
		= \sum_{k=j}^{r - q + j}\binom{r-q}{k-j}\Dcal_k,
$$

	as required.
\end{proof}

With this proposition in mind, we are ready to proceed to the first Representation Theorem.

\begin{thm}\label{thm1}
	Let $\vec{c}\in\DD{q}$ and $w_{\vec{c}}$ be a probability function satisfying Px. Then $w_{\vec{c}}$ is a member of
	a ULi with IP family $\W = \{w_{\vec{d}_r}\|\vec{d}_r\in\DD{r}\}$ if and only if each entry $c_i$ of
	$\vec{c}$ is of the form
	\begin{gather}
		c_i = \int_{[0,1]} x^{\gamma(i)} (1-x)^{q-\gamma(i)}\,d\rho(x)\label{eq2.3}
	\end{gather}
	for some normalized $\sigma$-additive measure $\rho$ on $[0,1]$.
\end{thm}

\begin{proof}
	We will use methods from Nonstandard Analysis working in a suitable nonstandard universe $^{*}V$, see for example \cite{Cutland}.
	The key idea to the proof is to marginalize some $w_{\vec{c}}$ on some infinite language
	to finite languages, rather than constructing extensions of some $w_{\vec{d}}$ on a finite language to each finite level.
	Suppose we have such a ULi with IP family $\W$ of probability functions, so for each $r\in\N$, we have some $w^{(r)}$ on $L_r$ in this family. By the Transfer Principle
	this holds for each $r\in\StarN$, so we can pick some nonstandard natural number $\nu\in\StarN\setminus\N$ and consider
	$w^{(\nu)}$. Now $w^{(\nu)}\restr \SL_r = w^{(r)}$ for each $r<\nu$, as these are members of the same ULi family and
	we can retrieve our original family $\W$ by looking at functions of the form $w^{(\nu)}\restr \SL_r$ for $r\in\N$, taking standard
	parts -- denoted as usual by ${}^{\circ}$ -- where necessary.

	In more detail let $^{*}V$ be a nonstandard universe that contains at least $\DD{q}$ for finite $q\in\N$, all probability functions
	$w_{\vec{b}}$ satisfying Px and everything else needed in this proof. Let $\nu\in\StarN$ be nonstandard and
	consider $\vec{b}\in\DD{\nu}$ such that $w_{\vec{b}}$ on $L_{\nu}$ satisfies Px. Assume that $\vec{\Bcal}$ is the
	alternative notation for $\vec{b}$  given by \eqref{eq2.1}. For each $q<\nu$,
	we can define a probability function on $L_q$ {\it in } $^{*}V$
satisfying Px by letting
	\begin{gather}
		\C_j = \sum_{\kappa=j}^{\nu-q+j}\binom{\nu-q}{\kappa-j}\Bcal_{\kappa}\label{eq2.4}
	\end{gather}
	for $j=0,\dotsc,q$. In general, this gives $\vec{c}\in\Star{\DD{q}}$, so we need to take the standard part of $\vec{c}$, denoted $\std\vec{c}$, to
	get a probability function $w_{\std\vec{c}}$ in $V$.
	
	We will first look at $\vec{\Bcal}$ when all weight is concentrated on a single $\Bcal_{\kappa}$, $0\leq\kappa\leq\nu$.
	Since we need to have $\sum_{\kappa=0}^{\nu}\binom{\nu}{\kappa}\Bcal_{\kappa} = 1$, we obtain
	\begin{gather*}
		\Bcal_\kappa = \binom{\nu}{\kappa}^{-1}.
	\end{gather*}
	Then we get for $0\leq j\leq q$
	\begin{align}
		\C_j = \binom{\nu-q}{\kappa-j}\Bcal_{\kappa} &= \binom{\nu-q}{\kappa-j}\cdot \binom{\nu}{\kappa}^{-1}\nonumber\\
		&=\frac{(\nu-q)!\cdot\kappa!\cdot(\nu-\kappa)!}{(\kappa-j)!\cdot(\nu-q-\kappa+j)!\cdot\nu!}\nonumber\\
		&=\frac{\kappa\cdot(\kappa-1)\dotsm(\kappa-j+1)\cdot(\nu-\kappa)\dotsm(\nu-\kappa-q+j+1)}
		{\nu\cdot(\nu-1)\dotsm(\nu-q+1)},\label{eq2.5}
	\end{align}
	thus leading to the standard part being
	\begin{gather}
		\std\C_j = \leftidx{^{\circ}}{\left(\left(\frac{\kappa}{\nu}\right)^j\cdot\left(1-\frac{\kappa}{\nu}\right)^{q-j}\right)} =
		 \leftidx{^{\circ}}{\left(\frac{\kappa}{\nu}\right)}^j\cdot \left(1 -  \leftidx{^{\circ}}{\left(\frac{\kappa}{\nu}\right)}\right)^{q-j}.\label{eq2.6}
	\end{gather}
	
	Now consider an arbitrary $\vec{\Bcal}=\<\Bcal_0,\dotsc,\Bcal_{\nu}\>$. Then for each $0\leq\kappa\leq\nu$ there exists
	$\gamma_{\kappa}\in\Star{[0,1]}$ such that we can write
	\begin{gather*}
		\Bcal_{\kappa} = \gamma_{\kappa}\cdot\binom{\nu}{\kappa}^{-1}.
	\end{gather*}
 Note that since
	\begin{gather*}
		\sum_{\kappa=0}^{\nu}\binom{\nu}{\kappa}\Bcal_{\kappa} = 1
		\intertext{we must have}
		\sum_{\kappa=0}^{\nu}\gamma_{\kappa} = 1.
	\end{gather*}

	Then using \eqref{eq2.5} we see that each summand in $\C_j$ will be of the form
	\begin{gather*}
		\gamma_{\kappa}\cdot\binom{\nu-q}{\kappa-j}\binom{\nu}{\kappa}^{-1},
	\end{gather*}
	thus $\std\C_j$ will become
	\begin{gather}
		\std\C_j = \leftidx{^{\circ}}{\left(\sum_{\kappa=j}^{\nu-q+j}\gamma_{\kappa}\cdot
		\binom{\nu-q}{\kappa-j}\binom{\nu}{\kappa}^{-1}\right)}.\label{eq2.7}
	\end{gather}
	Since we are only interested in the standard part, we can add the finitely many summands for $\kappa=0,\dotsc,j-1,\nu-q+j+1,\dotsc,\nu$ without changing
	$\std\C_j$ (assuming that $0<j<q$), as we have
	\begin{align*}
		&\leftidx{^{\circ}}{\left(\sum_{\kappa=0}^{\nu}\gamma_{\kappa}\cdot\binom{\nu-q}{\kappa-j}\binom{\nu}{\kappa}^{-1}\right)} - \std\C_j \\
		&= \leftidx{^{\circ}}{\left(\sum_{\kappa=0}^{j-1}\gamma_{\kappa}\cdot\binom{\nu-q}{\kappa-j}\binom{\nu}{\kappa}^{-1}\right)}
		+ \leftidx{^{\circ}}{\left(\sum_{\kappa=\nu-q+j+1}^{\nu}\gamma_{\kappa}\cdot\binom{\nu-q}{\kappa-j}\binom{\nu}{\kappa}^{-1}\right)}\\
		&= \sum_{\kappa=0}^{j-1}\leftidx{^{\circ}}{\left(\gamma_{\kappa}\cdot\binom{\nu-q}{\kappa-j}\binom{\nu}{\kappa}^{-1}\right)}
		+ \leftidx{^{\circ}}{\left(\sum_{\kappa=\nu-q+j+1}^{\nu}\gamma_{\kappa}\cdot\binom{\nu-q}{\kappa-j}\binom{\nu}{\kappa}^{-1}\right)}\\
		& = 0+0,
	\end{align*}
	because for $\kappa\in\{0,\dotsc,j-1,\nu-q+j+1,\dotsc,\nu\}$, either $\std(\kappa/\nu)=0$ or $\std(1-\kappa/\nu)=0$, so the first and last sum vanish as each
	consists of finitely many terms. Note that in case $j=0,q$, either the first or the second summand is empty, and therefore we can apply the same argument for $j=0,q$
	as well, giving
	\begin{gather}
		\std\C_j = \leftidx{^{\circ}}{\left(\sum_{\kappa=0}^{\nu}\gamma_{\kappa}\cdot\binom{\nu-q}{\kappa-j}\binom{\nu}{\kappa}^{-1}\right)}\label{eq2.8}
	\end{gather}
	for $j\in\{0,\dotsc,q\}$.
	
	Now let $N = \{0,\dotsc,\nu\}$ and (in $^{*}V$ of course)
 let $\mu$ be the Loeb counting measure on $N$
	(see example (1), section 2 in \cite{Cutland}). Then we can write \eqref{eq2.8} 
 as
	\begin{gather}\label{eq2.85}
		\std\C_j = \leftidx{^{\circ}}{\int_{N}\Bigr. \gamma_\kappa \cdot\binom{\nu-q}{\kappa-j}\binom{\nu}{\kappa}^{-1}\,d\mu(\kappa)}.
	\end{gather}
	
	Let $\mu'$ be the discrete measure on $\Star{[0,1]}$ which for $\kappa \in N$ gives the point
	$\kappa/\nu$ measure $\gamma_\kappa.$
	Then we get
	\begin{gather}
		\int_{N} \gamma_\kappa \cdot\binom{\nu-q}{\kappa-j}\cdot \binom{\nu}{\kappa}^{-1}\,d\mu(\kappa)
		= \int_{\Star{[0,1]}}\binom{\nu-q}{x \cdot \nu-j}\cdot \binom{\nu}{x \cdot \nu}^{-1} \,d\mu'(x).\label{eq2.9}
	\end{gather}

	Now let $\rho$ be the measure {\it in } $V$ on $[0,1]$ which for a Borel subset $A$ of $[0,1]$ gives
	\begin{gather}
		\rho(A) = \std \mu'(^{*}A).
	\end{gather}
	By well known results from Loeb Measure Theory, see for example \cite{Cutland},
	\begin{gather}\label{eq2.95}
		\leftidx{^{\circ}}{\int_{\Star{[0,1]}}\binom{\nu-q}{x \cdot \nu-j}\cdot \binom{\nu}{x \cdot \nu}^{-1} \,d\mu'(x)}
		= \int_{[0,1]} \leftidx{^{\circ}}{\left(\binom{\nu-q}{x \cdot \nu-j}\cdot \binom{\nu}{x \cdot \nu}^{-1}\right)} \, d\rho(x).
	\end{gather}
	Combining \eqref{eq2.6},\eqref{eq2.85},\eqref{eq2.9},\eqref{eq2.95} now gives that 
	\begin{gather}
		\std\C_j = \int_{[0,1]} x^j\cdot (1-x)^{q-j}\,d\rho(x)\label{eq2.12}
	\end{gather}
	We obtain a $\vec{c}\in\DD{q}$ by letting
	\begin{gather*}
		\vec{c} = \<\std\C_0,\std\C_1,\dotsc, \std\C_1,\dotsc, \std\C_{q-1},\dotsc, \std \C_{q-1},\std \C_q \>.
	\end{gather*}
	As we can marginalize $\vec{b}$ in the above way to any $r\in\N$, we obtain that given a family of functions $\{w_{\vec{d}_r} \| \vec{d}_r\in\DD{r}\}$
	such that each $\vec{d}_r$ is obtained by marginalizing some $\vec{b}\in\DD{\nu}$ and therefore satisfies \eqref{eq2.3}, this
	family satisfies Unary Language Invariance.
	
	For the converse it is straightforward to check that  any $w_{\vec{c}}$  for which all the $c_i$ in $\vec{c}$ are of the form \eqref{eq2.12} does satisfy ULi,
	the required family member on $L_r$ being obtained simply by changing $q$ to $r$ with the same measure $\rho$.
\end{proof}

However, as the following example will show, the probability functions of the form $w_{\vec{c}}$ satisfying ULi with IP are not the building blocks
that generate all probability functions satisfying ULi:

\begin{example}\label{ex1}
	Let $c_0^{L_2}$ be the probability function on $L_2$ given by
	\begin{gather*}
		c_0^{L_2} = 4^{-1}\left( w_{\<1,0,0,0\>} + w_{\<0,1,0,0\>} + w_{\<0,0,1,0\>} + w_{\<0,0,0,1\>}\right).
	\end{gather*}
	Then $c_0^{L_2}$ satisfies ULi
 as it is a member of Carnap's Continuum of Inductive
	Methods (see e.g. \cite{ParisVencovskaBook}).  
	However, both $\<0,1,0,0\>$ and $\<0,0,1,0\>$ are not of the form \eqref{eq2.3}, and thus $c_0^{L_2}$ shows that we
	cannot have a Representation Theorem for $w$ satisfying ULi of the form
	\begin{gather*}
		w = \int_{\DD{q}} w_{\vec{x}}\,\,d\mu(\vec{x})
	\end{gather*}
	with $\mu$ giving all weight to $\vec{c}$ of the form \eqref{eq2.3}.
\end{example}

\section{The Representation Theorem for $w$ satisfying ULi}

In the previous section, we used a probability function satisfying Px + IP
on the infinite language $L_\nu$ to construct a language invariant family by
marginalizing to each finite level.

In this section we shall instead derive a representation theorem for just ULi by using an arbitrary
state description $\Upsilon$ of $L_\nu$ to construct a probability function satisfying
Px by averaging over all permutations of predicates, similarly to the definition
of $c_0^{L_2}$ in \thref{ex1}.

Let $\Upsilon(P_1,\dotsc,P_\nu,a_1,\dotsc,a_\nu)$ be the state description of $L_\nu$ given by
\begin{gather*}
	\Upsilon(P_1,\dotsc,P_\nu,a_1,\dotsc,a_\nu) = \bigwedge_{i=1}^\nu\bigwedge_{j=1}^\nu P_i^{\epsilon_{i,j}}(a_j).
\end{gather*}
Then we can represent $\Upsilon$ by the $\nu\times\nu$ - matrix
\begin{gather}
	\begin{pmatrix}
		\epsilon_{1,1}&\epsilon_{1,2}&\dotsb&\epsilon_{1,\nu}\\
		\epsilon_{2,1}&\epsilon_{2,2}&\dotsb&\epsilon_{2,\nu}\\
		\vdots&\vdots&\ddots&\vdots\\
		\epsilon_{\nu,1}&\epsilon_{\nu,2}&\dotsb&\epsilon_{\nu,\nu}
	\end{pmatrix}.\label{eq2.14}
\end{gather}

Now consider the $q\times\nu$ - matrix $\Psi$ where the $j$'th row of $\Psi$ is the
$i_j$'th row of $\Upsilon$, for some $i_1,\dotsc,i_q\in\{1,\dotsc,\nu\}$, not necessarily distinct.
 Then  we can similarly think of $\Psi$ as a state description $\Psi(a_1, \ldots, a_\nu)$ of $L_q$.  So each column
of $\Psi$ represents an atom of $L_q$, and we obtain $\vec{c}\in\Star{\DD{q}}$ by letting
\begin{gather*}
	c_i = \frac{|\{j\|\Psi\models\al_i(a_j)\}|}{\nu}.
\end{gather*}
We thus obtain for each $\<i_1,\dotsc,i_q\>$ with $1\leq i_1,\dotsc,i_q\leq \nu$ some
$w_{\vec{c}}$ for $\vec{c}\in\Star{\DD{q}}$, which we shall denote by $w^\Upsilon_{\<i_1,\dotsc,i_q\>}$.

We can now define the functions that we will then use to prove the representation theorem for
general ULi functions.

\begin{defi}
	Let $\Upsilon(P_1,\dotsc,P_\nu,a_1,\dotsc,a_\nu)$ be a state description
	of $L_\nu$ for $\nu$ distinct constants. Let $L=L_q$ for some finite $q$.
	For $i_1,\dotsc,i_q\in\{1,\dotsc,\nu\}$, not necessarily distinct, let $w^\Upsilon_{\<i_1,\dotsc,i_q\>}$ be given as above.
	
	Define the function $\nabla^L_\Upsilon$	on $\SL$ by
	\begin{gather*}
		\nabla^L_\Upsilon = \sum_{e:\{1,\dotsc,q\}\rightarrow\{1,\dotsc,\nu\}}\frac{1}{\nu^q}w^\Upsilon_{\<e(1),\dotsc,e(q)\>}.
	\end{gather*}
\end{defi}

Instead of just marginalizing to the first $q$ rows, as we did in the case of $w_{\vec{c}}$,
$\nabla^L_\Upsilon$ now also averages over all permutations of the predicates. One can think of
this as picking $q$ rows from the matrix representing $\Upsilon$ \emph{with replacement} to obtain the predicates
$P_1,\dotsc,P_q$ of $L_q$.

Before our next result we need to recall another principle, see \cite{Hill}, \cite{ParisVencovskaBook}.

{\bfseries The Weak Irrelevance Principle, WIP}\\
{\itshape A probability function $w$ on $\SL$ satisfies \emph{Weak Irrelevance} if whenever
$\theta, \phi \in\QFSL$ have no constants nor predicates in common then
\begin{gather*}
	w(\theta \wedge \phi) = w(\theta)\cdot w(\phi).
\end{gather*}}

\begin{thm}\label{thmNabla}
	Let $\Upsilon(P_1,\dotsc,P_\nu,a_1,\dotsc,a_\nu)$ be a state description of $L_\nu$ and let
	$L=L_q$. Then the function $\std\nabla^L_\Upsilon$ is (can be extended to) a probability function on $\SL$ satisfying ULi + WIP.
\end{thm}

\begin{proof}
	From the definition of $\nabla^L_\Upsilon$ it is obvious that $\std\nabla^L_\Upsilon$ is a probability function satisfying Ex.
	
	For Px, let $\sigma$ be a permutation of the predicates of $L$. Then we obtain
	\begin{align*}
		\std\nabla^L_\Upsilon(\sigma\Theta) &= \leftidx{^{\circ}}{\left[\sum_{e:\{1,\dotsc,q\}\rightarrow\Upsilon}
				\frac{1}{\nu^q}\cdot w^\Upsilon_{\<e(1),\dotsc,e(q)\>}(\sigma\Theta)\right]}\\
			&= \leftidx{^{\circ}}{\left[\sum_{e:\{1,\dotsc,q\}\rightarrow\Upsilon}
				\frac{1}{\nu^q}\cdot w^\Upsilon_{\<e(\sigma^{-1}(1)),\dotsc,e(\sigma^{-1}(q))\>}(\Theta)\right]},
			\intertext{since $\sigma$ permutes the predicates of $L$,}
			&= \leftidx{^{\circ}}{\left[\sum_{e\circ\sigma^{-1}:\{1,\dotsc,q\}\rightarrow\Upsilon}
				\frac{1}{\nu^q}\cdot w^\Upsilon_{\<e\circ\sigma^{-1}(1),\dotsc,e\circ\sigma^{-1}(q)\>}(\Theta)\right]}\\
			&= \leftidx{^{\circ}}{\left[\sum_{e':\{1,\dotsc,q\}\rightarrow\Upsilon}
				\frac{1}{\nu^q}\cdot w^\Upsilon_{\<e'(1),\dotsc,e'(q)\>}(\Theta)\right]} = \std\nabla^L_\Upsilon(\Theta).
	\end{align*}
	
	To show that ULi holds, notice that for $\Theta(a_1,\dotsc,a_n)$ the state description
	\begin{gather*}
		\Theta(a_1,\dotsc,a_n) = \bigwedge_{j=1}^n\al_{h_j}(a_j),
	\end{gather*}
	we obtain on $L_{q+1}$,
	\begin{gather*}
		\Theta(a_1,\dotsc,a_n) = \bigvee_{\epsilon_1,\dotsc,\epsilon_n\in\{0,1\}}\bigwedge_{j=1}^n\al_{h_j}^{\epsilon_j}(a_j),
	\end{gather*}
	where
	\begin{gather*}
		\al_{h_j}^{\epsilon_j}(x) = \al_{h_j}(x)\wedge P_{q+1}^{\epsilon_j}(x).
	\end{gather*}
	We obtain
	\begin{align*}
		\std\nabla^{L_{q+1}}_\Upsilon&(\Theta)\\
		&= \sum_{\epsilon_1,\dotsc,\epsilon_n\in\{0,1\}}
		\std\nabla^{L_{q+1}}_\Upsilon\left(\bigvee_{\epsilon_1,\dotsc,\epsilon_n\in\{0,1\}}\bigwedge_{j=1}^n\al_{h_j}^{\epsilon_j}\right)\\
		&= \sum_{\epsilon_1,\dotsc,\epsilon_n\in\{0,1\}}
		\leftidx{^{\circ}}{\left[\sum_{e:\{1,\dotsc,q+1\}\rightarrow\{1,\dotsc,\nu\}}\frac{1}{\nu^{q+1}}
			 w^\Upsilon_{\<e(1),\dotsc,e(q+1)\>}\left(\bigvee_{\epsilon_1,\dotsc,\epsilon_n\in\{0,1\}}\bigwedge_{j=1}^n\al_{h_j}^{\epsilon_j}\right)\right]}\\
		&= \leftidx{^{\circ}}{\left[\sum_{e:\{1,\dotsc,q+1\}\rightarrow\{1,\dotsc,\nu\}}\frac{1}{\nu^{q+1}}
			\sum_{\epsilon_1,\dotsc,\epsilon_n\in\{0,1\}} w^\Upsilon_{\<e(1),\dotsc,e(q+1)\>}\left(\bigvee_{\epsilon_1,\dotsc,\epsilon_n\in\{0,1\}}\bigwedge_{j=1}^n\al_{h_j}^{\epsilon_j}\right)\right]}\\
		&=\leftidx{^{\circ}}{\left[\sum_{e':\{1,\dotsc,q\}\rightarrow\{1,\dotsc,\nu\}}\frac{1}{\nu^q}\cdot\right.}\\
		&\qquad\qquad \left.\sum_{f:\{1\}\rightarrow\{1,\dotsc,\nu\}}\frac{1}{\nu}
			\sum_{\epsilon_1,\dotsc,\epsilon_n\in\{0,1\}} w^\Upsilon_{\<e'(1),\dotsc,e'(q),f(1)\>}\left(\bigvee_{\epsilon_1,\dotsc,\epsilon_n\in\{0,1\}}\bigwedge_{j=1}^n\al_{h_j}^{\epsilon_j}\right)\right],
	\end{align*}
	where
	\begin{gather*}
		e(i) = \begin{cases}
					e'(i)&\text{if $i\in\{1,\dotsc,q\}$,}\\
					f(1)&\text{if $i=q+1$}.
				\end{cases}
	\end{gather*}

	It now remains to show that
	\begin{gather}
		\sum_{f:\{1\}\rightarrow\{1,\dotsc,\nu\}}\frac{1}{\nu}
			\sum_{\epsilon_1,\dotsc,\epsilon_n\in\{0,1\}} w^\Upsilon_{\<e'(1),\dotsc,e'(q),f(1))\>}\left(\bigvee_{\epsilon_1,\dotsc,\epsilon_n\in\{0,1\}}\bigwedge_{j=1}^n\al_{h_j}^{\epsilon_j}\right) = w^\Upsilon_{\<e'(1),\dotsc,e'(q)\>}(\Theta)\label{eq2.15}
	\end{gather}
	for arbitrary $e':\{1,\dotsc,q\}\rightarrow\hat{\Upsilon}$. There are $\vec{c}\in\Star{\DD{q}}$,
	$\vec{d}\in\Star{\DD{q+1}}$ such that
	\begin{align*}
		w^\Upsilon_{\<e'(1),\dotsc,e'(q)\>} = w_{\vec{c}},\\
		w^\Upsilon_{\<e'(1),\dotsc,e'(q),f(1)\>} = w_{\vec{d}}.
	\end{align*}
	Given $\be_j$ an atom of $L_{q+1}$, there is a unique atom $\al_i$ of $L_q$ and a
	unique $\epsilon\in\{0,1\}$ such that
	\begin{gather*}
		\be_j = \al_i^{\epsilon}.
	\end{gather*}
	Thus, we can unambiguously write $d_j = c_i^{\epsilon}$ for these $i$, $\epsilon$. We then
	obtain
	\begin{align}
		\sum_{\epsilon_1,\dotsc,\epsilon_n\in\{0,1\}} w^\Upsilon_{\<e'(1),\dotsc,e'(q),f(1))\>}\left(\bigvee_{\epsilon_1,\dotsc,\epsilon_n\in\{0,1\}}\bigwedge_{j=1}^n\al_{h_j}^{\epsilon_j}\right) &= \sum_{\epsilon_1,\dotsc,\epsilon_n\in\{0,1\}} \prod_{j=1}^n c_{h_j}^{\epsilon_j}\nonumber\\
		&= \prod_{j=1}^n (c_{h_j}^0 + c_{h_j}^1).\label{eq2.16}
	\end{align}
	Since by picking row $f(1)$ as the $q+1$'st row we partition the occurrences of the atom
	$\al_j$ of $L_q$ obtained by picking rows $e'(1),\dotsc,e'(q)$ into occurrences of the atoms
	$\al_j^1$ and $\al_j^0$ of $L_{q+1}$, and this is the only way in which we obtain these
	atoms, we must have $c_i^0 + c_i^1 = c_i$ for each $i\in\{1,\dotsc,2^q\}$. Thus \eqref{eq2.16}
	gives
	\begin{gather*}
		\prod_{j=1}^n (c_{h_j}^0 + c_{h_j}^1) = \prod_{j=1}^n c_{h_j} = w_{\<e'(1),\dotsc,e'(q)\>}(\Theta).
	\end{gather*}
	The equation \eqref{eq2.15} now follows.
	
	It remains to show that Weak Irrelevance holds for $\std\nabla^L_\Upsilon$. Let $\theta(a_1,\dotsc,a_m)$,\\
	$\phi(a_{m+1},\dotsc,a_{m+n})$ be state descriptions of $L$ having no constant or predicates in common. We can assume that
	$\theta\in\QFSL^1$, $\phi\in\QFSL^2$, where $L^1\cap L^2=\emptyset$ and $L^1\cup L^2=L$. Let $\al_i$ range over the atoms of $L^1$, $\be_j$ over the
	atoms of $L^2$. Then we obtain in $L^1$ and $L^2$, respectively,
	\begin{align*}
		\theta(a_1,\dotsc,a_m) &= \bigwedge_{i=1}^m \al_{h_i}(a_i),\\
		\phi(a_{m+1},\dotsc,a_{m+n}) &= \bigwedge_{j=1}^n \be_{g_j}(a_{m+j}).
	\end{align*}
	Suppose that $L^1 = \{P_1,\dotsc,P_p\}$, $L^2=\{P_{p+1},\dotsc,P_{p+r}\}$. Then we obtain in $L$
	\begin{align*}
		\theta(a_1,\dotsc,a_m) &= \bigvee_{1\leq s_1,\dotsc,s_m\leq 2^r}\bigwedge_{i=1}^m \al_{h_i}(a_i)\wedge\be_{s_i}(a_i),\\
		\phi(a_{m+1},\dotsc,a_{m+n}) &= \bigvee_{1\leq t_1,\dotsc,t_n\leq 2^p}\bigwedge_{j=1}^n \al_{t_j}(a_{m+j})\wedge\be_{g_j}(a_{m+j}),
	\end{align*}
	and by ULi for $\std\nabla^L_\Upsilon$,
	\begin{align}
		\std\nabla^{L^1}_\Upsilon(\theta) &= \std\nabla^{L}_\Upsilon\left(\bigvee_{1\leq s_1,\dotsc,s_m\leq 2^r}\bigwedge_{i=1}^m \al_{h_i}\wedge\be_{s_i}\right),\label{eq2.17}\\
		\std\nabla^{L^2}_\Upsilon(\phi) &= \std\nabla^{L}_\Upsilon\left(\bigvee_{1\leq t_1,\dotsc,t_n\leq 2^p}\bigwedge_{j=1}^n \al_{t_j}\wedge\be_{g_j}\right)\label{eq2.18}.
	\end{align}
	Now for $\theta\wedge\phi$, we obtain in $L$
	\begin{align*}
		\std\nabla^L_\Upsilon&(\theta\wedge\phi)\\
			&= \std\nabla^L_\Upsilon\left(\bigvee_{1\leq s_1,\dotsc,s_m\leq 2^r}\bigvee_{1\leq t_1,\dotsc,t_n\leq 2^p}\left(\bigwedge_{i=1}^n\al_{h_i}\wedge\be_{s_i}\right)\wedge\left(\bigwedge_{j=1}^n\al_{t_j}\wedge\be_{g_j}\right)\right)\\
			&= \sum_{1\leq s_1,\dotsc,s_m\leq 2^r}\sum_{1\leq t_1,\dotsc,t_n\leq 2^p}\leftidx{^{\circ}}{\left[\sum_{e:\{1,\dotsc,q\}\rightarrow\{1,\dotsc,\nu\}}\frac{1}{\nu^q}
			\cdot\right.}\\
				&\qquad\qquad \left.w^\Upsilon_{\<e(1),\dotsc,e(q)\>}\left(\left(\bigwedge_{i=1}^n\al_{h_i}\wedge\be_{s_i}\right)\wedge\left(\bigwedge_{j=1}^n\al_{t_j}\wedge\be_{g_j}\right)\right)\right]\\
			&= \sum_{1\leq s_1,\dotsc,s_m\leq 2^r}\sum_{1\leq t_1,\dotsc,t_n\leq 2^p}\leftidx{^{\circ}}{\left[\sum_{e:\{1,\dotsc,q\}\rightarrow\{1,\dotsc,\nu\}}\frac{1}{\nu^q}
			\cdot w^\Upsilon_{\<e(1),\dotsc,e(q)\>}\left(\bigwedge_{i=1}^n\al_{h_i}\wedge\be_{s_i}\right)\right.}\\
			&\qquad\qquad\left.\cdot w^\Upsilon_{\<e(1),\dotsc,e(q)\>}\left(\bigwedge_{j=1}^n\al_{t_j}\wedge\be_{g_j}\right)\right],
		\intertext{by IP for $w^\Upsilon_{\<e(1),\dotsc,e(q)\>}$,}
			&= \left(\sum_{1\leq s_1,\dotsc,s_m\leq 2^r}\leftidx{^{\circ}}{\left[\sum_{e:\{1,\dotsc,q\}\rightarrow\{1,\dotsc,\nu\}}\frac{1}{\nu^q}
					\cdot w^\Upsilon_{\<e(1),\dotsc,e(q)\>}\left(\bigwedge_{i=1}^n\al_{h_i}\wedge\be_{s_i}\right)\right]}\right)\\
			&\qquad	\cdot \left(\sum_{1\leq t_1,\dotsc,t_n\leq 2^p}\leftidx{^{\circ}}{\left[\sum_{e:\{1,\dotsc,q\}\rightarrow\{1,\dotsc,\nu\}}\frac{1}{\nu^q}
					\cdot w^\Upsilon_{\<e(1),\dotsc,e(q)\>}\left(\bigwedge_{j=1}^n\al_{t_j}\wedge\be_{g_j}\right)\right]}\right)\\
			&= \left(\sum_{1\leq s_1,\dotsc,s_m\leq 2^r}\std\nabla^L_\Upsilon\left(\bigwedge_{i=1}^m \al_{h_i}\wedge\be_{s_i}\right)\right)
				\cdot\left(\sum_{1\leq t_1,\dotsc,t_n\leq 2^p}\std\nabla^L_\Upsilon\left(\bigwedge_{j=1}^n\al_{t_j}\wedge\be_{h_j}\wedge\right)\right)\\
			&= \std\nabla^L_\Upsilon(\theta) \cdot \std\nabla^L_\Upsilon(\phi),
	\end{align*}
	by \eqref{eq2.17} and \eqref{eq2.18}.
\end{proof}

\begin{thm}\label{RepThm}
	Let $w$ be a probability function on $L=L_q$. Then $w$ satisfies ULi if and only if there exists some normalized $\sigma$-additive measure $\rho$ such that
	\begin{gather}
		w = \int \std\nabla^L_\Upsilon\,\,d\rho(\Upsilon).\label{eq2.19}
	\end{gather}
\end{thm}

\begin{proof}
	By \thref{thmNabla}, it is straightforward to see that any $w$ in the form \eqref{eq2.19} satisfies ULi, as it is a convex combination of ULi functions.

	For the other direction, suppose $w$ satisfied ULi. Then there is an extension $w^{L_\nu}$ of $w$ to $L_\nu$ and we obtain for $\Theta(a_1,\dotsc,a_n)$ a state
	description of $L$,
	\begin{gather}
		w(\Theta) = \sum_{\substack{\Phi(a_1,\dotsc,a_{\nu})\\\Phi\models\Theta}}w^{L_\nu}(\Phi),\label{eq2.20}
	\end{gather}
	where $\Phi$ ranges over the state descriptions of $L_\nu$. For a state description\\
	$\Upsilon(P_1,\dotsc,P_\nu,a_1,\dotsc,a_\nu)$, let
	\begin{gather*}
		\bar{\Upsilon} = \{\Upsilon(P_{\sigma(1)},\dotsc,P_{\sigma(\nu)},a_{\tau(1)},\dotsc,a_{\tau(\nu)}\| \sigma, \tau \text{ are permutations of }\{1,\dotsc,\nu\}\}.
	\end{gather*}
	Note that the sets $\bar{\Upsilon}$ partition the set of state descriptions of $L_\nu$. We can  now write \eqref{eq2.20} as
	\begin{align*}
		w(\Theta) &= \sum_{\bar{\Upsilon}}\sum_{\substack{\Phi\in\bar{\Upsilon}\\\Phi\models\Theta}}w^{L_\nu}(\Phi)\\
			&= \sum_{\bar{\Upsilon}}\frac{|\{\Phi\in\bar{\Upsilon}\|\Phi\models\Theta\}|}{|\bar{\Upsilon}|}w^{L_\nu}\left(\bigvee\bar{\Upsilon}\right),
	\end{align*}
	as $w^{L_\nu}$ is clearly constant on $\bar{\Upsilon}$ since it satisfies Px (and Ex).
	
	Now the ratio
	\begin{gather*}
		\frac{|\{\Phi\in\bar{\Upsilon}\|\Phi\models\Theta\}|}{|\bar{\Upsilon}|}
	\end{gather*}
	is equal to the probability that by randomly picking distinct predicates $P_{i_1},\dotsc,P_{i_q}$ and constants $a_{j_1},\dotsc,a_{j_n}$,
	we have that
	\begin{gather*}
		\Upsilon \models \sigma\Theta(a_{j_1},\dotsc,a_{j_n}),
	\end{gather*}
	where $\sigma$ is (an initial segment of) the permutation of predicates of $L_\nu$ with $\sigma(k) = i_k$ for $k\in\{1,\dotsc,q\}$.
	
	Note that with our definition of $\nabla^L_\Upsilon$, we allow the same row to be picked
	multiple times, so not all picks of rows represent a permutation of the predicates. Thus the difference between the probabilities given by $\nabla^L_\Upsilon$ and
	the above ratio is the difference between picking rows of $\Upsilon$ with and without replacement. However, since the probability of picking the same row twice
	is infinitesimal, it will disappear when taking standard parts.
	
	Thus we obtain
	\begin{gather*}
		\leftidx{^{\circ}}{\left(\frac{|\{\Phi\in\bar{\Upsilon}\|\Phi\models\Theta\}|}{|\bar{\Upsilon}|}\right)} = \std\nabla^L_\Upsilon(\Theta).
	\end{gather*}
	Now taking $\mu$ to be the measure on the $\bar{\Upsilon}$ given by $w^{L_\nu}$, we obtain
	\begin{gather*}
		\sum_{\bar{\Upsilon}}\frac{|\{\Phi\in\bar{\Upsilon}\|\Phi\models\Theta\}|}{|\bar{\Upsilon}|}w^{L_\nu}\left(\bigvee\bar{\Upsilon}\right)
			= \int \frac{|\{\Phi\in\bar{\Upsilon}\|\Phi\models\Theta\}|}{|\bar{\Upsilon}|}\,d\mu(\bar{\Upsilon}).
	\end{gather*}
	Taking standard parts, we obtain 
	\begin{align*}
		\leftidx{^{\circ}}{\int \frac{|\{\Phi\in\bar{\Upsilon}\|\Phi\models\Theta\}|}{|\bar{\Upsilon}|}\,d\mu(\bar{\Upsilon})} &= \int \leftidx{^{\circ}}{\left(\frac{|\{\Phi\in\bar{\Upsilon}\|\Phi\models\Theta\}|}{|\bar{\Upsilon}|}\right)}\,d\rho(\bar{\Upsilon})\\
		&= \int \std\nabla^L_\Upsilon\,d\rho(\bar{\Upsilon}),
	\end{align*}
	where $\rho$ is the Loeb measure given by the nonstandard measure $\mu$.
\end{proof}

Since $\std \nabla^L_{\Upsilon}$ satisfies WIP  we obtain the following theorem.

\begin{thm}
	The $\std\nabla^L_{\Upsilon}$ are the only functions satisfying ULi with WIP.
\end{thm}
	
\begin{proof}
	We follow essentially the proof for the analogous theorem for Atom Exchangeability, given in \cite{ParisVencovska11}.
	
	Let $w$ be a probability function satisfying  ULi with WIP. Let $\theta\in\QFSL$. Extend
	$w$ to $w'$ on some language $L'$ large enough so that we can permute the predicates and constants in
	$\theta$ to obtain $\theta'$ with no predicates nor constants in common with $\theta$. We can achieve this by picking
	$w'$ on $L'$ in the same ULi family as $w$, giving $w'\restr\SL = w$ and guaranteeing WIP for $w'$. By Px for $w'$
	we then have $w'(\theta)=w'(\theta')$. Now we clearly obtain
	\begin{align*}
		0 &= 2(w'(\theta\wedge\theta')-w'(\theta)\cdot w'(\theta'))\\
		&= \int\std\nabla^{L'}_\Psi(\theta\wedge\theta')\,d\mu(\Psi) - 2 \int\std\nabla^{L'}_\Psi(\theta)\,d\mu(\Psi)\cdot \int\std\nabla^{L'}_\Phi(\theta')\,d\mu(\Phi)\\
		&\qquad + \int\std\nabla^{L'}_\Phi(\theta\wedge\theta')\,d\mu(\Phi)\\
		&= \int\int\left( \std\nabla^{L'}_\Psi(\theta)^2 - 2\std\nabla^{L'}_\Psi(\theta)\cdot\std\nabla^{L'}_\Phi(\theta) + \std\nabla^{L'}_\Phi(\theta)^2\right)\,d\mu(\Psi)\,d\mu(\Phi)\\
		&= \int\int\left( \std\nabla^{L'}_\Psi(\theta) - \std\nabla^{L'}_\Phi(\theta)\right)^2\,d\mu(\Psi)\,d\mu(\Phi),
	\end{align*}
	using the Representation Theorem. Certainly, since the function under the integral is non-negative, there must be a measure
	$1$ set such that $\std\nabla^{L'}_\Psi$ is constant on this set for each $\theta\in\QFSL$, giving $w'=\std\nabla^{L'}_\Psi$
	for any $\Psi$ in this set. Since $w'\restr\SL = w$, i.e. $w= \std\nabla^{L'}_\Psi\restr\SL$, marginalizing $w'$ to $L$ yields
	$w = \std\nabla^L_\Psi$, as required.
\end{proof}

\section{A General Representation Theorem}

In the case of Atom Exchangeability (Ax) (see e.g. \cite[chapter 33]{ParisVencovskaBook}), we have a theorem stating that each $w$ satisfying Ax can be represented
as a difference of scaled ULi functions with Ax. In this section, we will prove the analogous version for Px. For the remainder of this section
we assume that $L=L_q$ for some $q\in\N$.

\begin{defi}
	Let $\vec{c}\in\DD{q}$. Let $\Sigma$ be the set of all permutations of atoms of $L$ that are induced by Px. Define the probability function $y_{\vec{c}}$
	on $\QFSL$ by
	\begin{gather*}
		y_{\vec{c}}(\Theta(a_1,\dotsc,a_n)) = \frac{1}{|\Sigma|} \sum_{\sigma\in\Sigma} w_{\sigma\vec{c}}(\Theta(a_1,\dotsc,a_n))
	\end{gather*}
	for state descriptions $\Theta(a_1,\dotsc,a_n)$ of $L$.
\end{defi}

Note that by definition, $y_{\vec{c}}$ satisfies Px. By a straightforward argument we obtain the following variation on de Finetti's Theorem:

\begin{thm}
	Let $w$ be a propability function on $\SL$ satisfying Px. Then there exists a normalized, $\sigma$-additive measure $\mu$ on the Borel
	sets of $\DD{q}$ such that
	\begin{gather}
		w\left(\bigwedge_{j=1}^n\al_{h_j}(a_j)\right) = \int_{\DD{q}} y_{\vec{c}}\left(\bigwedge_{j=1}^n\al_{h_j}(a_j)\right)\,d\mu(\vec{c}).\label{eq2.35}
	\end{gather}
	Conversely, given such a measure $\mu$, the function $w$ defined by \eqref{eq2.35} satisfies Px.
\end{thm}

The key to obtaining the desired General Representation Theorem will therefore involve finding a uniform representation of the building blocks $y_{\vec{c}}$ in terms
of a difference of ULi functions. The $\std\nabla^L_\Upsilon$ functions used for this proof will have a specific characterization that deserves a slightly
different notation. Since at this point, we will be working in the usual standard universe again, we will drop the standard part symbol $^{\circ}$
from the notation and assume that all $\nabla^L_\Upsilon$ from now on are given in their standard form.

Recalling the definition of $\nabla^L_\Upsilon$ note that for fixed $e:\{1,\dotsc,q\}\rightarrow\{1,\dotsc,\nu\}$, the function $w^\Upsilon_{\<e(1),\dotsc,e(q)\>}$ is
given by the $q\times\nu$ - matrix with the $i$'th row identical to the $e(i)$'th row of $\Upsilon$. Also, since with $w^\Upsilon_{\<e(1),\dotsc,e(q)\>}$ we also
have all the $w^\Upsilon_{\<\sigma(e(1)),\dotsc,\sigma(e(q))\>}$ for $\sigma$ ranging over the permutations of the predicates of $L$ occurring in
$\nabla^L_\Upsilon$, we see that this function is a convex combination of functions of the form $y_{\vec{c}}$.

We can now arrange $\nabla^L_\Upsilon$ to contain a copy of $y_{\vec{c}}$ for a given $\vec{c}\in\DD{q}$ as follows: Let $\Phi$ be the state description
represented by the matrix
\begin{gather*}
	\begin{pmatrix}
		\vline & & \vline & \vline & & \vline & & \vline & & \vline\\
		\al_1 & \dotsb & \al_1 & \al_2 & \dotsb & \al_2 & \dotsb & \al_{2^q} & \dotsb & \al_{2^q}\\
		\vline & & \vline & \vline & & \vline & & \vline & & \vline
	\end{pmatrix},
\end{gather*}
where $\al_i$ occurs $[c_i\cdot \nu]$ times. Now let $\p_1,\dotsc,\p_q \geq 0$ be such that $\sum_{i=1}^q \p_i = 1$ and let $\Upsilon$ be the $\nu\times\nu$ - matrix
containing $[\p_i\cdot\nu]$ copies of the $i$'th row of $\Phi$, for each $i$, and fill the remaining rows with arbitrary copies of rows from $\Phi$. Then
$\nabla^L_\Upsilon$ certainly contains a copy of $y_{\vec{c}}$.

With this in mind, we can modify the notation of $\nabla^L_\Upsilon$ to
\begin{gather*}
	\vec{\p}\nabla^L_\Upsilon
\end{gather*}
for $\vec{\p}=\<\p_1,\dotsc,\p_q\>$ to indicate that $\Upsilon$ contains only $q$ distinct rows, occurring with the frequency given by $\vec{\p}$.
We will write $\vec{\p}\nabla^L_{\Upsilon(\vec{c})}$ to indicate that $\Upsilon$ arises from $\vec{c}\in\DD{q}$ in this manner.

We can represent $\vec{\p}\nabla^L_{\Upsilon(\vec{c})}$ in terms of $y_{\vec{c}}$ as follows.  Let $K= \{\vec{n} \in \N^q \,\|\, \sum_{i=1}^q n_i = q\}$, so
$\vec{n}\in K$ represents the choices of picking rows from $\Upsilon$. Then we obtain the representation
\begin{gather}
	\vec{\p}\nabla^L_{\Upsilon(\vec{c})} = \sum_{\vec{n}\in K} \prod_{i=1}^q \p_i^{n_i} (n_1,\dotsc,n_q)!\, y_{\vec{c}_{\vec{n}}},\label{eq2.37}
\end{gather}
where $\vec{c}_{\vec{n}}$ results from picking rows according to $\vec{n}$ and (as standard)
$$ (n_1,\dotsc,n_q)! = \frac{(n_1 + n_2 + \ldots + n_q)!}{n_1!\, n_2! \ldots n_q!}= \binom{q}{n_1,\dotsc,n_q}.$$
Note that we need this multinomial coefficient here since $\vec{\p}\nabla^L_{\Upsilon(\vec{c})}$ is in fact a sum of $w_{\vec{e}}$, and although each of the
$w_{\vec{e}}$ occurring in $y_{\vec{c}}$ occurs, the normalizing constant exists only implicitly in $\vec{\p}\nabla^L_{\Upsilon(\vec{c})}$.
With this notation in mind, we can prove the first step needed to show the desired theorem.

\begin{lem}\label{lem2.s}
	Let $\vec{c}\in\DD{q}$. Then there exist $\lambda\geq 0$ and probability functions $w_1$, $w_2$ satisfying ULi such that
	\begin{gather*}
		y_{\vec{c}} = (1+\lambda) w_1 - \lambda w_2.
	\end{gather*}
\end{lem}

\begin{proof}
	Fix $\vec{c}\in\DD{q}$. As demonstrated in the discussion above, we can easily find $\nabla^L_\Upsilon$ with $y_{\vec{c}}$ occurring in it,
	amongst other instances of $y_{\vec{e}}$. Thus, the problem reduces to finding a way to remove all of these other instances via ULi functions.
	
	To this end, suppose that for each $\vec{m} \in K$ we have $\vec{\p}_{\vec{m}}\nabla^L_{\Upsilon(\vec{c})}$ 
 such that $\Upsilon$ is the state
	description obtained from $w_{\vec{c}}$ by the method discussed above. Then, since the representations of the form \eqref{eq2.37}
	of these functions only differ in the coefficients of the $y_{\vec{e}}$ occurring we obtain the equation
	\begin{gather}
		\begin{pmatrix}
			\vdots\\
			\vec{\p}_{\vec{m}}\nabla^L_{\Upsilon(\vec{c})}\\
			\vdots
		\end{pmatrix}
		= A\cdot
		\begin{pmatrix}
			\vdots\\
			(m_1,\dotsc,m_q)!\, y_{\vec{c}_{\vec{m}}}\\
			\vdots
		\end{pmatrix},\label{eq2.38}
	\end{gather}
	where $A$ is the $K\times K$-matrix with entry $\<\vec{m},\vec{n}\>$ being $\prod_{k=1}^q \p_{\vec{m},k}^{n_k}$. It suffices now to show that we can pick
	the $\vec{\p}_{\vec{m}}$ such that $A$ is regular. For suppose this is the case. Then we obtain from \eqref{eq2.38} the equation
	\begin{gather}
		A^{-1}
		\begin{pmatrix}
			\vdots\\
			\vec{\p}_{\vec{m}}\nabla^L_{\Upsilon(\vec{c})}\\
			\vdots
		\end{pmatrix}
		=
		\begin{pmatrix}
			\vdots\\
			(m_1,\dotsc,m_q)!\,  y_{\vec{c}_{\vec{m}}}\\
			\vdots
		\end{pmatrix}.\label{eq2.39}
	\end{gather}
	Suppose $A^{-1} = (b_{\vec{n},\vec{m}})_{\vec{n},\vec{m}\in K}$. Then for $\vec{n}=\<1,1,\dotsc,1\>$ we obtain
	\begin{gather*}
		y_{\vec{c}} = \frac{1}{(n_1,\dotsc,n_q)!} \sum_{\vec{m}\in K} b_{\vec{n},\vec{m}}\vec{\p}_{\vec{m}}\nabla^L_{\Upsilon(\vec{c})} =
		\frac{1}{q!} \sum_{\vec{m}\in K} b_{\vec{n},\vec{m}}\vec{\p}_{\vec{m}}\nabla^L_{\Upsilon(\vec{c})},
	\end{gather*}
	and by collecting the functions with positive coefficients in the linear combination on the right-hand side, we obtain constants
	$\gamma, \lambda\geq 0$, {\it independent of } $\vec{c}$, such that\footnote{Note that we can safely assume $\lambda\neq 0$, since if $\lambda=0$,
		then the $y_{\vec{c}}$ in question would already satisfy ULi, and therefore already has the desired representation by the Representation Theorem for ULi.
		We also trivially have $\gamma\neq 0$, since $y_{\vec{c}}$ is a probability function for any $\vec{c}\in\DD{q}$.}
	\begin{gather*}
		\frac{1}{q!} \sum_{\vec{m}\in K} b_{k,\vec{m}}\vec{\p}_{\vec{m}}\nabla^L_{\Upsilon(\vec{c})} = \gamma w_1 - \lambda w_2,
	\end{gather*}
	with $w_1$, $w_2$ convex combinations of ULi functions. Since this gives the probability function $y_{\vec{c}}$, we must have
	\begin{gather*}
		1 = y_{\vec{c}}(\top) = \gamma w_1(\top) - \lambda w_2(\top) = \gamma - \lambda,
	\end{gather*}
	and thus $\gamma = 1 + \lambda$.
	
	It remains to show that the $\vec{\p}_{\vec{m}}$ can be chosen such that $A$ is regular. For this, we will show the following by induction on $j$:\\
	Let $1\leq i_1 < i_2 < \dotsb < i_j \leq r$ and let $A_{\<i_1,\dotsc,i_j\>}$ be the $j\times j$ sub-matrix of $A$
	obtained by taking the $i_1,\dotsc,i_j$'th rows and columns of $A$. Then there is a choice of the $\vec{\p}_{\vec{m}_k}$, $k=i_1,\dotsc,i_j$ such that
	$A_{\<i_1,\dotsc,i_j\>}$ is regular.
	
	For $j=1$, this is trivial. Suppose $j=n+1$ for some $n\geq 1$ and consider $A_{\<i_1,\dotsc,i_j\>}$.
	For a given $\vec{m}\in K$, the polynomial $\prod_{j=1}^q x_j^{m_j}$ takes its maximum value on $\DD{q}$ at $x_j = m_j/q$. Fix an enumeration of $K$. There exists
	$\vec{m}_{i_k} = \<m_{i_k,1},\dotsc,m_{i_k,q}\>$ such that
	\begin{gather*}
		\prod_{s=1}^q \left(\frac{m_{i_k,s}}{q}\right)^{m_{i_k,s}} > \prod_{s=1}^q \left(\frac{m_{i_k,s}}{q}\right)^{m_{i_j,s}}
	\end{gather*}
	for all $j\neq k$. For if not, then
	\begin{gather*}
		\prod_{s=1}^q \left(\frac{m_{i_k,s}}{q}\right)^{m_{i_k,s}} \leq  \prod_{s=1}^q \left(\frac{m_{i_k,s}}{q}\right)^{m_{i_j,s}} < \prod_{s=1}^q
			\left(\frac{m_{i_j,s}}{q}\right)^{m_{i_j,s}}
	\end{gather*}
	for some $j\neq k$, and continuing in this way we arrive at a contradiction.
	
	By the inductive hypothesis, there exists a choice of the $\vec{\p}_{\vec{m}_s}$, $s\in\{i_1,\dotsc,i_j\}\setminus\{i_k\}$ such that the sub-matrix
	$A_{\<i_1,\dotsc,i_{k-1},i_{k+1},\dotsc,i_j\>}$ is regular. Thinking of the $\p_{\vec{m}_{i_k},s}$ for the moment as unknowns we obtain for the determinant of
	$A_{\<i_1,\dotsc,i_j\>}$ an expression of the form
	\begin{multline}
		\det(A_{\<i_1,\dotsc,i_j\>}) =\\ \pm \prod_{s=1}^q \p_{\vec{m}_{i_k},s}^{m_{i_k,s}}\cdot\det(A_{\<i_1,\dotsc,i_{k-1},i_{k+1},\dotsc,i_j\>})
			+ \sum_{t\in\{i_1,\dotsc,i_j\}\setminus\{i_k\}}\prod_{s=1}^q \p_{\vec{m}_{i_k},s}^{m_{t,s}}\cdot\left(\pm\det(A_t)\right),\label{eq2.40}
	\end{multline}
	(for some choices of $\pm$) where the $A_t$ are the corresponding sub-matrices of $A_{\<i_1,\dotsc,i_j\>}$. Now picking $\p_{\vec{m}_{i_k,s}} = (m_{i_k,s}/q)^g$ for large 
	enough $g>0$, the term
	\begin{gather*}
		\prod_{s=1}^q \p_{\vec{m}_{i_k},s}^{m_{i_k,s}}\cdot\det(A_{\<i_1,\dotsc,i_{k-1},i_{k+1},\dotsc,i_j\>})
	\end{gather*}
	becomes the dominant term of \eqref{eq2.40}, giving that $\det(A_{\<i_1,\dotsc,i_j\>})\neq 0$, as certainly\\
	$\prod_{s=1}^q \p_{\vec{m}_{i_k,s}}^{n_{i_k,s}}> 0$ and $\det(A_{\<i_1,\dotsc,i_{k-1},i_{k+1},\dotsc,i_j\>})\neq 0$ by the inductive hypothesis.

	Note that using this procedure we in general obtain $\vec{\p}_{\vec{m}}$ with entries $\p_{\vec{m}_i,j}$ not summing to $1$. In that case, we can pick $\vec{\p'}_{\vec{m}}$
	such that
	\begin{gather*}
		\p_{\vec{m}_i,j}' = \frac{\p_{\vec{m}_i,j}}{\sum_{s=1}^q \p_{\vec{m}_i,s}}
	\end{gather*}
	for each $\vec{m}\in K$. Then the matrix $A'$ with entries $\prod_{s=1}^q {\p'}_{\vec{m}_i,s}^{n_j,s}$ is regular just if $A$ is, and
	the $\vec{\p'}_{\vec{m}}$ have the desired properties.
\end{proof}

Using this lemma, we can now prove the desired theorem.

\begin{thm}[General Representation Theorem for $w$ satisfying Px]
	Let $w$\\ be a probability function on $\SL$ satisfying Px. Then there exist $\lambda \geq 0$ and probability functions $w_1$, $w_2$ satisfying ULi
	such that
	\begin{gather*}
		w = (1+\lambda) w_1 - \lambda w_2.
	\end{gather*}
\end{thm}

\begin{proof}
	Let $w$ be a probability function on $\SL$ satisfying Px. By the Representation Theorem for Px, we have that $w$ has a representation
	\begin{gather}
		w = \int_{\DD{q}} y_{\vec{c}}\,\,d\mu(\vec{c})\label{eq2.50}
	\end{gather}
	for some measure $\mu$, and by \thref{lem2.s}, we have, for a fixed $\lambda \geq 0$, a representation
	\begin{gather*}
		y_{\vec{c}} = (1 + \lambda) w_{1_{\vec{c}}} - \lambda w_{2_{\vec{c}}}
	\end{gather*}
	for each $\vec{c}\in\DD{q}$. Now applying this to the representation \eqref{eq2.50}, we obtain
	\begin{align*}
		w &= \int_{\DD{q}} (1 + \lambda) w_{1_{\vec{c}}} - \lambda w_{2_{\vec{c}}}\,d\mu(\vec{c})\\
			&= \int_{\DD{q}} (1 + \lambda) w_{1_{\vec{c}}}\,d\mu(\vec{c}) - \int_{\DD{q}}\lambda w_{2_{\vec{c}}}\,d\mu(\vec{c}) \\
&=(1+\lambda)w_{1}  - \lambda w_{2},
	\end{align*}
for
$$ w_1 = \int_{\DD{q}} w_{1_{\vec{c}}}\,d\mu(\vec{c}), ~~~~ w_2 = \int_{\DD{q}} w_{2_{\vec{c}}}\,d\mu(\vec{c}),$$
as required.
\end{proof}

\section{Conclusion}

With \thref{RepThm}, we have shown that the building blocks for probability functions satisfying Unary Language Invariance all satisfy Weak Irrelevance, and
that in fact these are the only ones that satisfy this principle. This is analogous to the situation with Atom Exchangeability, Ax, and its generalization to Polyadic Pure Inductive Logic, Spectrum
Exchangeability, see \cite{ParisVencovskaBook}.
This analogy also extends to the General Representation Theorem, stating that each probability function satisfying Px is a scaled difference of probability
functions satisfying ULi (see \cite{Heid}). 

Throughout this paper we have worked in the conventional {\itshape Unary} Pure Inductive Logic. Recently however there has been a rapid development of
{\itshape Polyadic} Pure Inductive Logic (again see \cite{ParisVencovskaBook}) and we anticipate that the Representation Theorem  for ULi functions can be
extended to the polyadic case, using the same methods as demonstrated above. A classification for probability functions on polyadic languages satisfying
Language Invariance would give rise to the question whether we can find a corresponding General Representation Theorem for the polyadic case as well.

\end{document}